\documentclass[11pt,twoside]{amsart}
\usepackage{amsmath}
\usepackage{amssymb}
\usepackage{a4wide}
\usepackage{latexsym}
\usepackage{epsfig}

\newcommand{\N}{\ensuremath{\mathbb{N}}}
\newcommand{\Z}{\ensuremath{\mathbb{Z}}}
\newcommand{\R}{\ensuremath{\mathbb{R}}}
\newcommand{\E}{\ensuremath{\mathbb{E}}}
\newcommand{\X}{\ensuremath{\mathbb{X}}}

\renewcommand{\H}{\ensuremath{\mathbb{H}}}
\newcommand{\C}{\ensuremath{\mathcal{C}}}

\newcommand{\Rr}{\ensuremath{{\mathcal R}}}
\renewcommand{\P}{\ensuremath{{\mathcal P}}}

\newcommand{\T}{\ensuremath{{\mathcal T}}}

\DeclareMathOperator{\Sym}{Sym}
\DeclareMathOperator{\supp}{supp}

\newtheorem{thm}{Theorem}[section]

\newtheorem{prop}[thm]{Proposition}

\newtheorem{corollary}[thm]{Corollary}

\newtheorem{definition}[thm]{Definition}

\begin{document}

\title{Properties of
  B\"or\"oczky tilings in high dimensional hyperbolic spaces}

\author{Nikolai Dolbilin}
\address{Steklov Mathematical Institute, Gubkin 8, 119991 Moscow GSP-1
  Russia }
\email{dolbilin@mi.ras.ru}

\author{Dirk Frettl\"oh}
\address{Fakult\"at f\"ur Mathematik, Universit\"at
  Bielefeld, 33501 Bielefeld,  Germany}
\email{dirk.frettloeh@math.uni-bielefeld.de}
\urladdr{http://www.math.uni-bielefeld.de/baake/frettloe}

\begin{abstract} In this paper we consider families of
  B\"or\"oczky tilings  in hyperbolic space in arbitrary
  dimension, study some basic properties and classify all possible
  symmetries. In particular, it is shown that these tilings are
  non-crystallographic, and that there are uncountably many tilings
  with a fixed prototile. 
\end{abstract}

\maketitle

\section{\bf Introduction}  \label{sec:intro}

Let ${\mathbb X}$ be either a Euclidean space, or a hyperbolic
space, or a spherical space. Consider compact subsets $T \subseteq
{\mathbb  X}$, such that $T$ is the closure of its interior.  A
collection of such sets $\{T_1,T_2,\ldots\}$ is called a {\em
tiling}, if the union of the $T_i$ is the whole space $\X$ and the
$T_i$ do not overlap, i.e., the interiors of the tiles are
pairwise disjoint. The $T_i$ are called {\em tiles} of the tiling.

In 1975, K.\ B\"or\"oczky published some ingenious constructions of
tilings in the hyperbolic plane $\H^2$ \cite{boe}. His aim was to
show that there is not such a natural definition of density in
the hyperbolic plane $\H^2$ as there is in the Euclidean plane
$\E^2$. A very similar tiling is is mentioned also in \cite{p},
wherefore these tilings are often attributed to R.\ Penrose. In
context of a local theorem for regular point sets (\cite{ddsg}, see
also \cite{do1}), M.I.\ Shtogrin realized that the B\"or\"oczky tilings 
are  not crystallographic. V.S.\ Makarov considered analogues of these 
tilings in $n$-dimensional hyperbolic space and paid attention to
the fact that for all $d \ge 2$, B\"or\"oczky-type prototiles never
admit isohedral tilings of $\H^d$ \cite{mak}. This is of interest in
the context of Hilbert's 18th problem. Let us give a simple
description of one of B\"or\"oczky's constructions, which follows
\cite{ftk}. We assume familiarity with basic terms and facts of
$d$-dimensional hyperbolic geometry, see for instance \cite{and} (for
$d=2$, with emphasis of length and area in $\H^2$) or \cite{rat} (for
$d \ge 2$).    

Let $\ell$ be a line in the hyperbolic plane of curvature
$\chi=-1$ (see Figure \ref{boetiles}). Place points $\{ X_i \, |
\, i \in \Z \}$ on $\ell$, such that the length of the line
segment $X_i X_{i+1}$ is $\ln 2$ for all $i \in \Z$. Draw
through every point $X_i$ a horocycle $E_i$ orthogonal to 
$\ell$, such that all the horocycles $E_i$ have a common ideal
point ${\mathcal O}$ at infinity. On each $E_i$, choose points
$X_i^j$, such that $X_i^0=X_i$ and the length of the arc $X_i^j
X_i^{j+1}$ is the same for all $i,j \in \Z$. Denote by $\ell_j$ the
line parallel to $\ell$, which intersects $E_0$ at $X_0^j$, and
intersects $\ell$ at the ideal  point ${\mathcal O}$.
(In particular: $\ell_0=\ell$.) Due to the choice $|X_i X_{i+1}|=
\ln 2$ and $\chi=-1$, the arc on $E_1$ between $\ell_0$ and
$\ell_1$ is twice the length as the arc on $E_0$ between
the same lines.

Let $H_0$ be a strip between horocycle $E_0$ and horocycle
$E_1$ (including $E_0$ and $E_1$ themselves), and $L_i$ a strip
between two parallel lines $\ell_i$ and $\ell_{i+1}$  (including
$\ell_i$ and $\ell_{i+1}$ themselves). All  intersections
$B_i:=H_0 \cap L_i$, $i\in \Z $, are pairwise congruent shapes.

These shapes tile the  plane \cite{boe},\cite{ftk}. Indeed, they
pave a horocyclic strip between horocycles $E_0$ and $E_1$. It is
clear that shape $B_1$ can be obtained from $B_0$ by means of 
a horocyclic turn $\tau$ about the ideal point $\mathcal O$ that
moves point $X_0\in E_0$ to point $X^1_0 \in E_0$. Any shape $B_i$
in the horocyclic strip $H_0$ is obtained as $\tau^i(B_0)$.
Now, let $\lambda$  be a shift along $\ell_0$ moving the strip $H_0$
to the strip $H_1$. This shift transfers the pavement of $H_0$ to a
pavement of $H_1$. Thus, a sequence of shifts $\lambda^j$ 
($j\in\Z$) along a line $\ell_0$ extends the pavement of the
horocyclic strip $H_0$ to the entire plane, resulting in one of
several possible  B\"or\"oczky tilings.

\begin{figure}[h] \label{boetiles}
\epsfig{file=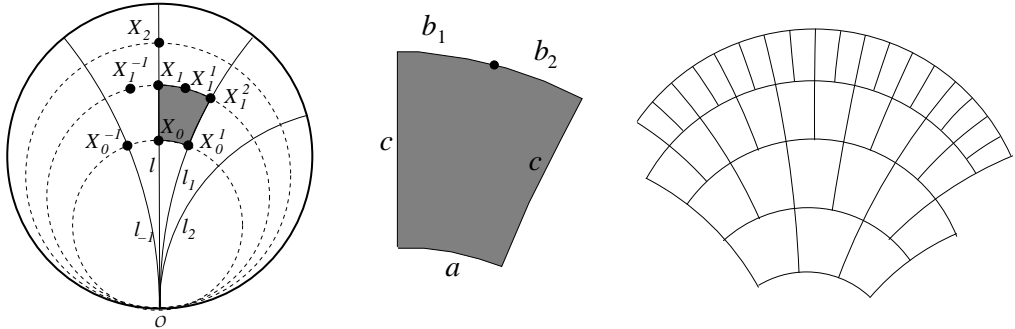}
\caption{The construction of the prototile in the Poincar\'e-disc
  model (left), a B\"or\"oczky pentagon $B$ with the types of the five
  edges indicated (centre), and a small
  part of a B\"or\"oczky tiling (right).}
\end{figure}

Let $B:=B_0=L \cap H_0$. The boundary of $B$ consists of two
straight line edges, labelled $c$ in Figure  \ref{boetiles}
(centre), and two arcs, a shorter one labelled $a$, and a longer arc
labelled $b_1,b_2$. The arcs are parts of horocycles. The
longer arc is subdivided in two halves $b_1$ and $b_2$. For
brevity, we will say that $B$ has three (horocyclic) edges $a,
b_1, b_2$ and two straight line edges $c$. Thus we regard $B$ as a
pentagon. Indeed, though the B\"or\"oczky pentagon $B$ is
neither bounded by straight line segments,  nor it is 
a convex  polygon, it suffices to replace all three  horocyclic
arcs by straight line edges in order to get a convex
polygonal tile. But there is no need to use this modification here.

On the $c$-edges, we introduce an
orientation from edge $a$ to edge $b_1$ or $b_2$ respectively. An
edge-to-edge tiling is a tiling in which any non-empty
intersection of tiles is either their common  edge or a vertex. An
edge-to-edge tiling of the hyperbolic plane by B\"or\"oczky
pentagons, respecting the orientations of $c$-edges, is called
a {\em B\"or\"oczky tiling}, or, for the sake of brevity, 
{\em  B-tiling}. 

A tiling is called {\em crystallographic}, if the symmetry group
of the tiling has compact fundamental domain, see
\cite{m}\footnote{Sometimes, a group is called crystallographic if
  its fundamental domain has finite volume  \cite{v}.}. As we have
already mentioned, all B-tilings are not
crystallo\-gra\-phic. In the next sections we will examine the
symmetry group of B-tilings in detail.

In Section \ref{h2}, we describe shortly the situation in the
hyperbolic plane $\H^2$. Some basic terms and tools are
introduced. Since Section \ref{h2} is a special case of Section
\ref{h4}, some proofs are omitted to Section \ref{h4}. 
In Section \ref{sec:crystthm} we state a necessary and sufficient
condition for a tiling to be crystallographic \cite{do1}. This
statement can be applied to B-tilings in $\H^2$.
In Section \ref{h4} we consider B\"or\"oczky's construction in
arbitrary dimension $d\ge 2$ and obtain some results on $B$-tilings
in $\H^d$. Section \ref{h5} uses these results and gives a complete
classification of the symmetry groups of these tilings. 

We denote $d$-dimensional hyperbolic space by $\H^d$, $d$-dimensional
Euclidean space by $\E^d$, and the set of positive integers by $\N$.

\section{\bf The B\"or\"oczky Tilings in $\H^2$} \label{h2}
In a B\"or\"oczky tiling every tile is surrounded by adjacent
tiles in essentially the same way: A $c$-edge meets always a
$c$-edge of an adjacent tile. A $b_1$-edge ($b_2$-edge resp.) meets an
$a$-edge of an adjacent tile. The $a$-edge of a tile  always touches
either the $b_1$-edge or the $b_2$-edge of an adjacent tile.

Therefore we consider in the following definitions two kinds of
subsets of tiles in a B-tiling, which are connected in
two different ways: either we move from tile to an adjacent
tile only across $c$-edges, or only across $a$-, $b_1$-,  or
$b_2$-edges.

\begin{definition} A {\em ring} in a B-tiling is a set
$(T_i)_{i \in \Z}$ of tiles ($T_i \ne T_j$ for $i \ne j$), such
that $T_i$ and $T_{i+1}$ share $c$-edges for all $i \in \Z$
(see Figure 2).
\end{definition}

\begin{definition}
Let $I\subset \Z$ be a set of consecutive numbers. A {\em
horocyclic path} is a sequence $(T_i)_{i \in I}$ of tiles ($T_i
\ne T_j$ for $i \ne j$), such that, for all $i \in \Z$, $T_i$ and
$T_{i+1}$ share an edge  which is not a $c$-edge (see Figure
2).
\end{definition}

We always require rings and horocyclic paths to contain no
loops, i.e., all tiles in the sequence are pairwise different. 
A ring is a sequence of tiles which is infinite in
both directions, shortly {\em biinfinite}. A ring forms a pavement
of a horocyclic strip bounded by two consecutive horocycles $E_i$
and $E_{i+1}$. A horocyclic path is either finite, or infinite in
one direction, or biinfinite.

If, in the case of horocyclic paths, an $a$-edge  of $T_i$ touches
a $b_1$- or  $b_2$-edge  of $T_{i+1}$, we say shortly: At this
position the horocyclic path goes {\em down}, otherwise we say it
goes {\em up}.
\begin{prop} Any horocyclic path in a B-tiling contains
  either only ups, or only downs, or is of the form
  down-down- $\cdots$-down-up-up- $\cdots$-up.
\label{noupdown} \end{prop}
\begin{proof} 
 There is only one way to pass from a tile through its $a$-edge to an
 adjacent  tile. So there is only one way to go
 down from a given tile. If we have in a horocyclic path the
 situation '$T_i$, up to $T_{i+1}$, down to $T_{i+2}$', it follows
 $T_i=T_{i+2}$. This situation is ruled out by the requirement, that
 all tiles in a horocyclic path are different.
So 'up-down' cannot happen, which leaves only the possibilities
mentioned in Proposition \ref{noupdown}. 
\end{proof}

\begin{figure} \epsfig{file=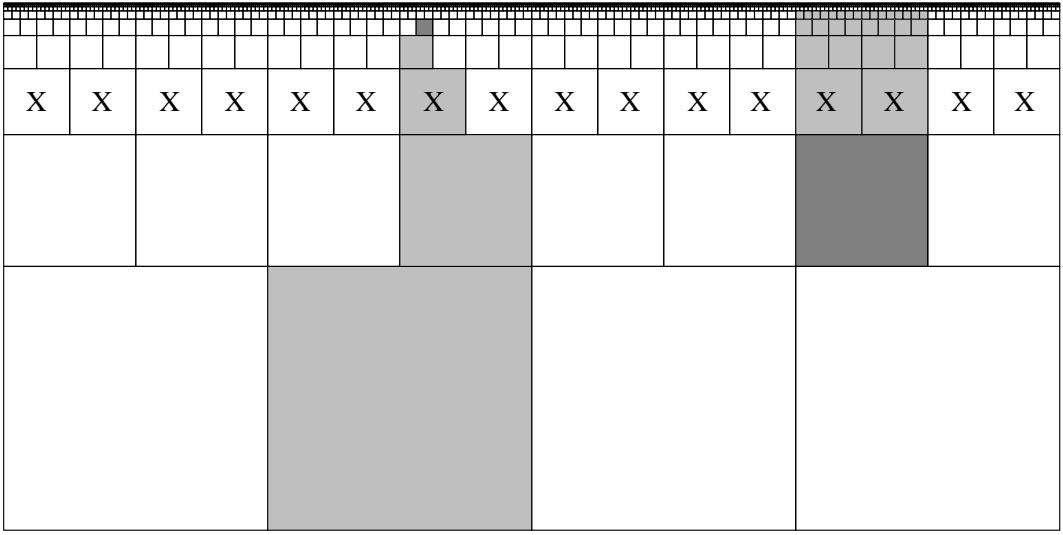}
\caption{A patch of the B\"or\"oczky tiling in the half-plane
  model. The tiles marked by an X show (a part of) a ring. The grey
  shaded tiles show: a horocyclic path of length five 
  (left, it is also the beginning of the tail of
  the dark tile on top); the tower on the dark tile (right, for the
  definition of a tower, cf. Section \ref{h4}).  }
  \label{bpathstowers}
\end{figure}

There are always two possibilities to go up from some tile $T$ in a
B-tiling. That means, there are $2^k$ different paths starting in $T$
and containing exactly $k$ ups (and no downs). Since there is
only one possibility to go down from any tile, there is only
one infinite horocyclic path of the form down-down-down-
$\cdots$ starting in $T$.

\begin{definition}
The unique infinite path of the form down-down-$\cdots $ starting at
a tile $T$ is called the {\em tail (of $T$)}.\label{taildef}
\end{definition}

If tiles $T$ and $T'$ and also tiles $T'$ and $T''$ are linked by
horocyclic paths $w$ and $w'$ respectively, then the tiles $T$
and $T''$ are also linked by a horocyclic path. Therefore, the
property of tiles to be linked by a horocyclic path is an
equivalence relation. Hence a set of tiles is partitioned into a
number of non-overlapping classes $\P _1, \P _2, \ldots$ called
pools.

\begin{definition}  \label{defpool} The set of all tiles  which can be 
  linked with some given tile by a horocyclic path forms a pool.
\end{definition}
It is clear that if $T\in \P$ then $t(T)\subset \P$.
\begin{definition} \label{defequi}
  Two tails $t(T)=(T=T_0,T_2,\ldots )$, $t(T')=(T'=T'_0,T'_2, \ldots )$
  in a B-tiling are called {\em   cofinally equivalent}, written $t(T)
  \sim t(T')$, if there are $m, n \in \N$ such that
  $T_{m+i}=T'_{n+i}$ for all $i \in \N$. 
\end{definition}
If $t(T)\sim t(T')$ then $T$ and $T'$ are linked by a horocyclic
path
\[ T=T_0, T_1, \ldots , T_m, T'_{n-1}, \ldots , T'_1, T'_0 = T', \]
hence $T$ and $T'$ belong to one pool. The inverse statement is also
obviously true. Thus, the following is true.
\begin{prop}
Tiles $T$ and $T'$ belong to a pool if and only if their tails are
cofinally equivalent.  \label{pooltail}
\end{prop}
A tail $(T_0,T_1,T_2,\ldots)$ gives rise to a sequence $s(T) =
(s_1,s_2,\ldots) \in   \{-1,1\}^{\N}$ in the following way: If the
$a$-edge  $a$ of $T_{i-1}$ coincides with the $b_1$-edge  of
$T_i$, then set $s_i:=1$, otherwise (if it coincides with the
$b_2$-edge) set $s_i:=-1$.

\begin{definition} Let $s = (s_1,s_2,\ldots), s' =
  (s'_1,s'_2,\ldots) \in \{-1,1\}^{\N}$.
\begin{itemize}
\item $s$ and $s'$ are {\em cofinally equivalent}, denoted by 
  $s \sim s'$, if there exist \\ $m, n \in \N$, such that
  $s_{m+i}=s'_{n+i}$ for all $i \in \N$. 
\item $s$ is {\em periodic (with period $k$)}, if there exists 
  $k \ge 2$ such that $s=(s_k,s_{k+1},\ldots)$.
\item $s$ is {\em cofinally periodic}, if $s$ is cofinally
equivalent to a periodic sequence.
\end{itemize}
\label{perdef} \end{definition}
\begin{prop} Every pool contains infinitely many tiles. Moreover,
  every intersection of any ring with any pool contains
  infinitely many tiles. \label{poolinfinit}
\end{prop}
\begin{proof}
The tail of any tile $T$ contains infinitely many tiles, all belonging to
the same pool. Moreover, from any tile there are starting $2^k$
horocyclic paths of length $k$ going upwards. The final tiles of these
paths are all in the same ring.

So, by going down from $T$ by $k$ steps and after that going up $k$
steps  gives us horocyclic paths to $2^k$ different tiles in the same
ring as $T$ and in the same pool as $T$. Since $k$ can be chosen
arbitrary large, the claim follows. \end{proof}
\begin{prop} In a B-tiling in $\H^2$ there is either one
  pool or two pools. \label{12pools}
\end{prop}
\begin{proof}
First of all, note the following. Let $x$ be an interior point
of some pool, and let $\ell'$ be a line such that $x \in \ell'$, and 
such that $\ell'$ contains the ideal point $\mathcal O$ at 
infinity (compare Section \ref{sec:intro}). Then $\ell'$
is contained entirely in the interior of the
pool. Therefore, if a point $x$ is on the boundary
of two pools then the line $\ell'$ entirely belongs to the
boundary of two pools.

Assume there is more than one pool. Consider some ring $R$.
Each of the pools has infinite intersection with $R$. Let $T$
and $T' \in R$ be adjacent tiles from different pools, say,
$\P_1$ and $\P_2$.  Then a line $\ell'$ containing the common
$c$-edge of $\P_1$ and $\P_2$ separates these pools. Moreover,
since a pool is a linearly connected set, these pools lie on
opposite sides each  of the line $\ell'$. Assume there is a third
pool $\P_3$. Let it lie on the same side of $\ell'$ as $\P_2$.
Since $\P_3\cap R \neq \varnothing$, there is some tile $T_3 \in
\P_3 \cap R$. There exists a line $\ell''$ which separates $\P_2$ and
$\P_3$. The pool has to lie between two parallel  lines $\ell'$
and $\ell''$. But this impossible because in this case the
intersection $\P\cap R$ is finite what contradicts Proposition
\ref{poolinfinit}.
\end{proof}
In particular, there are two pools if and only if there are tiles 
$T, T' \in \T$, whose tails have sequences $s(T)=(1,1,1,\ldots)$ and
$s(T')=(-1,-1,-1,\ldots)$.

\begin{prop} In any B-tiling holds: $t(T) \sim t(T')
  \Leftrightarrow s(T) \sim s(T')$
\end{prop}
\begin{proof} One direction ($\Rightarrow$) is clear from the
  construction of $s(T)$ out of $t(T)$.

The other direction: If there is only one pool, all tiles are
cofinally equivalent, and we are done. If there are two pools
$\P_1 \ne \P_2$, we know that these pools corresponds to
sequences $(-1,-1,-1,\ldots)$ and $(1,1,1,\ldots)$. Therefore, if
$s(T) \sim s(T')$, then $T$ and $T'$ belong to the same pool:
$t(T) \sim t(T')$.
\end{proof}

Let $\Sym(\T)$ be the symmetry group of $\T$, i.e., the set of all
isometries $\varphi$ where $\varphi(\T)=\T$. The group of order two is
denoted by $\C_2$. 
\begin{thm} The symmetry group $\Sym(\T)$ of any B-tiling in $\H^2$ is
\begin{itemize}
\item isomorphic to $\Z \times \C_2$ in the case of two pools, 
\item isomorphic to $\Z$ in the case of one
pool and $s(T)$  periodic for some $T$,
\item trivial else.
\end{itemize} \label{thm2dim}
\end{thm}
\begin{proof} (in Section \ref{h4}) \end{proof}
If there is one pool and $s(T)$ is periodic for some $T$, it can
happen, that $\Sym(\T)$ contains shifts along a line, and also
'glide-reflections', i.e. a reflection followed by a shift along a
line. The latter is the case, if $s=(s_1,s_2,\ldots)$ has a period
$k$ as in Definition \ref{perdef}, where $k$ is an even number,
and if holds:
\begin{equation} \label{essper}
s=(-s_{k/2+1},-s_{k/2+2}, \ldots).
\end{equation}
In the following, we mean by {\em essential period}
$\frac{k}{2}$, if (\ref{essper}) holds, otherwise $k$.

{\em Example:} A tiling $\T$ with a periodic sequence
\[ s(T)=(1,1,-1,-1,1,1,-1,-1,\ldots) \]
 has period 4, and essential period 2. $\T$
has a symmetry $\varphi$, where $\varphi$ is a glide-reflection.
The action of the reflection on $s(T)$ gives
$(-1,-1,1,1,-1,-1,1,1,\ldots)$. This is followed by a shift along
an edge of type $C$ along two tiles which gives
\[ (1,1,-1,-1,1,1,-1,-1,\ldots)=s(T). \]
\begin{corollary}
The fundamental domain of $\Sym(\T)$ is
\begin{itemize}
\item one half ring in the case of two pools, \item the union of
$k$ rings in the case of one pool and $s(T)$ for some $T$
  periodic, where $k$ is the essential period of $s(T)$,
\item $\H^2$ else.
\end{itemize}
\end{corollary}

\section{\bf B\"or\"oczky tilings are non-crystallographic}
\label{sec:crystthm}

Let $\T$ be a face-to-face tiling in $\R^d$ or $\H^d$. 
\begin{definition} The 0-corona $C_0(T)$ of a tile $T$ is $T$ itself.
The $k$-corona $C_k(T)$ of $T$ is the complex of all
tiles of $\T$ which 
have a common $(d-1)$-face with some $T' \in C_{k-1}(T)$.
\end{definition}
Figure \ref{bcorona} shows a tile $T$ in a B-tiling together with its
first and second coronae. Note that, for $T \ne T'$, the coronae
$C_k(T)$ and $C_k(T')$ can coincide as complexes. Nevertheless, the
corona denoted by $C_k(T)$ is considered as a corona about the centre
$T$, whereas corona $C_k(T')$ considered  as a corona about $T'$.

\begin{definition}
Coronae $C_k(T)$  and $C_k(T')$ are considered as congruent if
there is an isometry that moves $T$ to $T'$ and $C_k(T)$ to
$C_k(T')$.\label{congrcoronae}
\end{definition}
Given a tiling $\T$, denote by $N_k$ the number of congruence classes
of $k$-coronae.

\begin{figure}[t] \epsfig{file=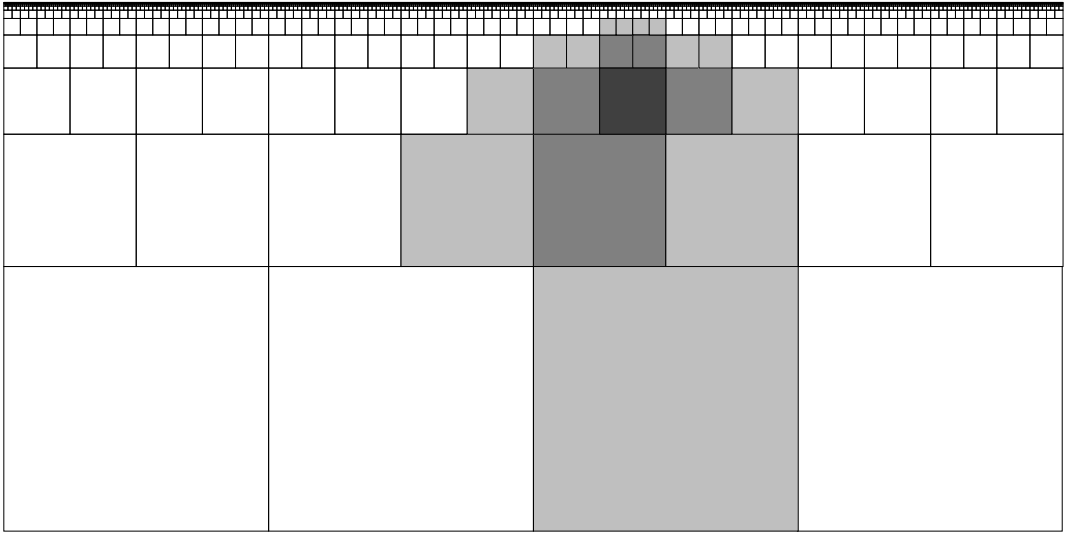} \label{bcorona}
\caption{The first corona (medium grey) and the second corona (all
  grey tiles) of the dark tile.}
\end{figure}

\begin{prop} The number $N_k$ of different $k$-coronae $(k \ge 1)$ in
  a  B-tiling is $2^{k-1}$,  up to isometries. \label{ncorona}
\end{prop}
\begin{proof} 
Let us enumerate all rings in a biinfinite way. Given a tile $T$, we
denote the ring containing $T$ by $\Rr_0$. All the
rings lying upward from $\Rr_0$ are enumerated as $\Rr_{-1}$,
$\Rr_{-2}$, $\Rr_{-3}, \ldots$. All rings going downward from
$\Rr_0$ are enumerated by positive integers $\Rr_1$, $\Rr_2$, $\Rr_3,
\ldots$.

It is easy to see that $C_k(T)\cap \Rr_i\neq \varnothing$ if and
only if $-k\leq i \leq k$. 
It is also easy to check that for any tiles $T$ and $T'\in \Rr_0$
the following complexes in $\T$
\[ C^-_k(T):=C_k(T)\cap (\bigcup_{i=-k}^0 \Rr_i ) \quad \mbox{and} 
\quad C^-_k(T'):=C_k(T')\cap (\bigcup_{i=-k}^0 \Rr_i )\]
are pairwise congruent. In other words: All $k$-coronae look the
same, if we only 
look at the rings $\Rr_0,\Rr_{-1},\Rr_{-2},\ldots$ (i.e., the
tiles 'above' $T$). Moreover, since  $T$ and $C^-_k(T)$ are both
mirror symmetric, this congruence can be realized by two
isometries both moving $T$ into $T'$. In addition, the same is
true also for any two tiles from a B-tiling.

The upper part of any corona is on one convex  side of a horocycle.
This part obviously contains a half plane cut of by a line
perpendicular to $\ell$. Thus,  all the tiles in $\T$  for any
$k\in \N$ have $k$-coronae whose larger parts are pairwise
congruent.

So, the upper part $C^-_k (T)$ of some $k$-corona $C_k(T)$ around
$T$ is uniquely determined independently of $T$. The entire  
corona $C_k(T)$ is determined completely (up to orientation)  by
the first $k$ members of the tail $t(T)$, or, what is equivalent, by 
the first numbers $s_1, s_2, \ldots , s_k$ in the sequence $s(T)$. Our 
next task is to show that in a given tiling for any finite
sequence of $k$ $\pm 1$'s there is a tile $T$ whose tail is
encoded exactly by this sequence.

Let us enumerate the tiles in ring $\Rr_0$ (the ring
containing $T$) by upper indices. Let $T^0:=T$ and let other tiles on
one side of $T^0$ be enumerated by positive $i$'s and on the other
side by negative $i$'s in a consecutive manner.

Let $s_k(T^0)=(s_1, s_2, \ldots , s_k)$. Then (see Figure
\ref{bpathstowers}) all
tiles $T^{2n}$ from ring $\Rr_0$ have the same first number $s_1$
in $s(T^{2n})$. And, in general, it is easy to see that every
$T^i\in \Rr_0$, where $i\equiv 0  \pmod{2^k} $,  has the same first
$k$ numbers in $s(T^i)$. It follows that in a biinfinite  sequence
$(T_i)\in \Rr_0$ each $2^k$-th tile has the same coronae encoded by
$s_k(T^0)$.

Consider in ring $\Rr_0$ a segment $T^0, T^1, T^2, \ldots \,
T^{2^k-1}$. Recall that $s_k(T^0)=(s_1, s_2, \ldots s_k)$. Now, one can
easily check that either $s_k(T^1)=(-s_1, s_2,s_3, \ldots s_k)$, or
$s_k(T^2)=(s_1, -s_2,s_3, \ldots s_k)$, compare for 
instance Figure \ref{bpathstowers}.

And, in general, take  $0\leq i\leq 2^k-1$. The number $i$ can be
represented in a unique way as
\[ i= \sum_{j=0}^{k-1}\delta_i 2^i, \]
where $\delta _i = 0 $ or 1. Then, as  one can easily  see,
\[ s_k(T^i)= \{(-1)^{\delta_i}s_i|\,  0\leq i\leq k \}. \]
Thus, any B-tiling contains $2^k$ different $k$-coronae, up
to orientation. \end{proof} 

Now we should emphasise that the number $N_k$ of coronae in a
B-tiling  is unbounded, as $k$ tends to infinity.
But in any crystallographic tiling
(independently on how it is defined, as with either compact or
cocompact  fundamental domain), the number of coronae classes is
always bounded: $N_k\leq m <\infty$, where $m$ is some fixed
number independent of $k$. Therefore we have obtained the following
result. 

\begin{thm} \label{bncryst}
All B-tilings are non-crystallographic.
\end{thm}

On the other side, this last theorem (in the context of Proposition
\ref{ncorona} and a local theory on crystallographic structures)
can  be considered as a consequence of the following theorem, which
appears in \cite{do1}. This theorem is a generalization of the Local
Theorem, see \cite{ddsg}.

\begin{thm} \label{crysthm}
  Let $\X^d$ be a Euclidean or hyperbolic space.
  A tiling $\T$ in $\X^d$ is crystallographic if and only if
  the following two conditions hold for some $k \ge 0$:
\begin{enumerate}
\item For the numbers $N_k$ of $k$-coronae in $\T$ holds: $N_{k+1} =
  N_k$, and $N_k$ is finite.
\item $S_{k+1}(i) = S_k(i)$ for $1 \le i \le N_k$,
\end{enumerate}
where $S_k(i)$ denotes the symmetry group of the $i$-th $k$-corona.
\end{thm}

Note that Condition (2) in the theorem makes sense only when
Condition (1) is fulfilled.
In a B-tiling, condition (2) is violated for $k=0$, and
condition (1) is violated for all $k\geq 1$. Thus Theorem
\ref{bncryst} is a particular case of the Local Theorem.

But in fact, the necessity part  of the Local Theorem is trivial
in contrast to a valuable sufficiency part. Indeed, if for any
$k\in \N$ at least one condition does not hold, this means that
$N_k\rightarrow \infty$ as $k\rightarrow \infty$. The latter one
implies that the tiling under consideration is non-crystallographic.

\section{\bf B\"or\"oczky Tilings in $\H^{d+1}$} \label{h4}

Let us give the construction of a
$d+1$-dimensional prototile. The construction in Section
\ref{sec:intro} is a special case of this one, with $d=1$.
Throughout this section, we will make strong use of the fact that
any $d$-dimensional horosphere $E$ in $\H^{d+1}$ is isometric to the
Euclidean space $\E^d$, see for instance \cite[\S 4.7]{rat}. Whenever
we consider a subset $A$ of a horosphere $E$, we will freely switch
between regarding $A$ as a subset of $\E^d$ and a subset of $E$. 

Let $\ell$ be a line in
$\H^{d+1}$. Choose a horosphere $E_0$ orthogonal to $\ell$. Let
$\Box'$ be a $d$-dimensional cube in $E_0$, centred at $E_0 \cap \ell$.
Let $H_1, \ldots , H_{2d}$ be hyperbolic hyperplanes, orthogonal
to $E_0$, such that each $H_i$ contains one of the $2d$ distinct
$(d-1)$-faces of the cube $\Box'$. Denote by $H^+_i$ the halfspace
defined by $H_i$ which contains the cube $\Box'$, and let
$C_{d+1}:=\bigcap _{i=1}^{2d} H^+_i$. Note that the intersection
$C_{d+1}\cap E_0 =\Box'$. Let $E_1$ be 
another horosphere, such that $E_1$ is concentric with $E_0$, $E_1$ is
not contained in the convex hull of $E_0$, and the distance of $E_0$
and $E_1$ is $\ln 2$. Then, $C_{d+1}\cap E_1$ is a $d$-dimensional
cube  of edge-length 2. Without loss of generality,  let this  
cube in $E_1$ be $\Box:= \{x=(x_1, \ldots,x_d) \in E_1 \, | \, -1
\le x_i \le 1 \}$. Divide $\Box$ into $2^d$ cubes
\[ \Box_{(\sigma_1, \sigma_2,\ldots,\sigma_d)} = \{ x \in \Box \, | \,
x_i\sigma_i \ge 0 \}, \; \sigma_i \in \{-1,1\}. \]
Let $L_0$ be the layer between $E_0$ and $E_1$. A 
{\em B\"or\"oczky prototile} in $\H^{d+1}$ is defined as
\[ B:= C_{d+1}  \cap L_0. \]

By construction, the prototile $B$ has $2^d+2d+1$ facets: one
'lower' facet (a unit Euclidean cube in $E_0$, denoted as
$a$-facet), $2^d$ 'upper' facets (unit Euclidean cubes in $E_1$,
denoted as $b$-facets), and $2d$ aside facets, denoted as
$c$-facets. One should mention that the B\"or\"oczky prototile is not 
a convex polyhedron. Moreover, it is not a polyhedron at all.
Though we call them facets, the lower and upper facets are
Euclidean $d$-cubes isometrically embedded into hyperbolic space
and they do not lie in hyperbolic hyperplanes. The $c$-facets
lie in hyperbolic hyperplanes, but for $d\ge2$ they are also not
$d$-dimensional polyhedra, because  their boundaries do not
consist of  hyperbolic polyhedra of dimension $d-1$. A $c$-facet
of a $d+1$-dimensional B\"or\"oczky prototile is a point set
which lies in a hyperbolic  hyperplane $H$ and is contained
between the intersections $E_0\cap H$ and $E_1\cap H$. Since
the intersections are also horocycles of dimension $d-1$, the 
$c$-facets are B\"or\"oczky prototiles of dimension $d-1$.

\begin{prop} Prototile $B$ admits  a face-to-face tiling $T$ by its
  copies. \label{tiling} 
\end{prop}

\begin{proof}
Take a horosphere $E_0$ and a unit cube $\Box'\subset E_0$. There
are parabolic turns  $g_i$ ($i=1,2,\ldots , d$) of $\H^{d+1}$ such
that their restrictions $g_i|_{_{E_0}}$ on $E_0$ are translations
of $E_0$ along edges of $\Box '$. The $g_i$ ($i=1,2,\ldots , d$)
span an Abelian group $G$ isomorphic to $\Z^d$. The orbit of the
B\"or\"oczky prototile under $G$ is a pavement of the layer between
$E_0$ and $E_1$ by copies of $B$.

Let $\tau$ be a shift of $\H^{d+1}$ along the line $\ell$ moving $E_0$
to $E_1$. Then, the shifts $\tau^ k$, $k\in \Z$, move the pavement of
the mentioned layer into all other layers, resulting in a tiling of
the whole space.
\end{proof}

We call any face-to-face tiling of $\H^{d+1}$ with prototile $B$ a 
{\em B\"or\"oczky tiling}, or shortly {\em B-tiling}.
Analogously to $\H^2$,  we call $a$- and $b$-facets of $B$  
{\em horospheric} facets. We should emphasise that if some horosphere
$E_0$ is fixed, then a unique sequence $(E_i)_{i \in \Z}$ of
horospheres orthogonal to $\ell$ is induced, provided the distance
along $\ell$ between any two consecutive horospheres $E_i$ and
$E_{i+1}$ is $\ln 2$. By construction of a B-tiling,
all $a$- and $b$-facets induce a tiling of each horosphere by
pairwise parallel unit cubes.

Let us state the definition of an analogue of a two-dimensional ring.
\begin{definition} 
Let $\T$ be a B-tiling in $\H^{d+1}$, and let $L_i$ be a layer of space
between horospheres $E_i$ and $E_{i+1}$. The set $\Rr_i$ of all tiles
from $\T$ lying in $L_i$ is called a {\em layer (of the tiling)}. 
\end{definition}
It is obvious that two tiles belong to one layer if and only if
there is some path $T_i$, $i\in I$, such that for any $i, i+1 \in I$,
the tiles $T_i$ and $T_{i+1}$ have a $c$-facet in common.

Let $I$ be a segment of $\Z$ of consecutive numbers. In analogy to
Section \ref{h2}, a sequence $(T_i)_{i \in I}$  of tiles is called
a horospheric path, if for all $i,i+1 \in I$ the tiles $T_i$ and
$T_{i+1}$ share a horospheric facet, and $T_i \ne T_j$ if $i\ne
j$. By construction, the single $a$-facet of any tile in some
B\"or\"oczky tiling in $\H^{d+1}$ coincides with one of $2^d$
$b$-facets of an adjacent tile.  To describe horospheric paths by
sequences, we use the alphabet
\[ {\mathcal A} = \{ \sigma \, | \, \sigma=(\sigma_1, \sigma_2, \ldots,
\sigma_d), \sigma_i = \pm 1 \}, \, |\mathcal A|=2^d \]

First of all, we emphasise that for  any tile  $B'$ from a
B-tiling, a natural bijection between its $2^d$ $b$-facets and the 
elements of $\mathcal{A}$ is uniquely determined, provided such a
bijection (between $b$-facets and elements of $\mathcal{A}$) is
already established for some tile $B$. Indeed, assume that a tile
$B'$ lies in the $i$-th layer $\Rr_i$. Then the bijection for $B$ is
canonically carried  by a single shift $\tau^i$ of the tile $B$
along $\ell$ into $\Rr_i$, followed by an appropriate translation
inside the layer $\Rr_i$.

Analogous to the two-dimensional case, any horospheric path 
$(T_i)_{i \in I}$ of the form down-down-$\cdots$  gives rise to a word 
$s(T)=(\sigma^{(i)})_{i \in  I}$ over  ${\mathcal A}$ in the following
way: If an $a$-facet of $T_i$ lies on a $b$-facet  of $T_{i+1}$, which
corresponds to the cube $\Box_{(\sigma_1,\ldots,\sigma_d)}$, we set
$\sigma^{(i)}:= (\sigma_1,\ldots,\sigma_d)$.

As in Section \ref{h2}, the tail $t(T)$ of a tile $T$ is the
unique infinite horospheric path $(T=T_0,T_1,\ldots)$ beginning in
$T$ of the form down-down-down-$\cdots$.

The definition of a pool and of the equivalence of tails goes exactly as
in Definition \ref{defpool} and Definition \ref{defequi} (replace
'horocyclic' by 'horospheric'). Moreover, the proofs of Propositions
\ref{noupdown} and \ref{poolinfinit} work in any dimension, so both
propositions are valid here. 

\begin{definition} For a tile $T \in \T$, let $W(T)$ be the set of all
tiles $T'$, such that $T$ can be linked with $T'$ by a horospheric
path $T,\ldots, T'$ of the form up-up-$\cdots$. We call $W(T)$ a 
{\em  tower}, or the {\em tower on $T$}. 
\end{definition}

For an example of a tower in $\H^2$, see Figure \ref{bpathstowers}.

\begin{prop}  \label {pooltower}
Let $\P \subset \T$ be a pool, $T \in \P$ and $t(T)=(T
=T_0,T_1,T_2,\ldots)$. Then 
 \[ W(T_i) \subset   W(T_{i+1}) \quad \mbox{and}  \]
\begin{equation} \label{eq:pooltower}
 \P = \bigcup_{T' \in   t(T)} W(T') = W(T_0) \cup W(T_1) \cup W(T_2)
\cup   \cdots.
\end{equation}
\end{prop}
\begin{proof}
Since a horospheric path $T_{i+1}, T_i$ is of the form 'up', any tile 
$T'\in W(T_i)$ can be linked with $T_{i+1}$ by a path of the form
up-up-$\cdots$, namely $T_{i+1}, \ldots, T'$. This implies $ W(T_i)
\subset  W(T_{i+1})$. 

Of course, all tiles $T' \in t(T)$ are in the same pool as $T$
because they are connected by a horospheric patch in $t(T)$. All
tiles in $W(T')$ are in the same pool as $T'$, therefore in the
same pool as $T$. This shows $\P \supseteq \bigcup_{T' \in t(T)}
W(T')$.

Let $T'' \in \P$. Then, by definition of a pool, exists a
horospheric path connecting $T$ and $T''$. By Proposition
\ref{noupdown}, this path is either of the form up-up-$\cdots$, or
down-down-$\cdots$, or down-down-$\cdots$-down-up-$\cdots$-up.

If this path $(T, \ldots, T'')$ is up-up-$\cdots$, then $T''\in W(T_i)$
for any $T_i\in t(T)$, thus $T''$ is contained in the right hand side
of \eqref{eq:pooltower}.

If the path $(T, \ldots, T'')$ is down-down-$\cdots$, then 
$T'' \in t(T)$, thus $T''$ is contained in the right hand side
of \eqref{eq:pooltower}.

If the path $(T, \ldots, T'')$ is down-down-$\cdots$-down-up-$\cdots$-up,
then this path contains a tile $\tilde{T}$, such that the path 
$(T, \ldots,  T'')$ consists of two paths: a down-down-$\cdots$-down path
$(T,\ldots,\tilde{T})$, and an up-up-$\cdots$ path 
$(\tilde{T},\ldots, T'')$. From here it follows that $\tilde{T}\in
t(T)$ and $T''\in W(\tilde{T})$, thus $T''$ is contained in the right
hand side of \eqref{eq:pooltower}.
\end{proof}

Let us now investigate the structure of a tower and a pool,
respectivley. Consider a pool $\P \subseteq \T$. Fix some horosphere 
$E_0$ and let a tile $T \in \P$ have  its $b$-facets  in
$E_0$. Obviously, the intersection
\[ u^{}_0(T):= \supp(W(T)) \cap E_0 \]
is exactly the union of  $b$-facets (which are unit cubes)  of
$T$. (Here, $\supp(W(T))$ denotes the support of $W(T)$, that is, the
union of all tiles in $W(T)$). 
This union is a $d$-dimensional cube $C_2$ of edge length 2.
We introduce in $E_0$, as in Euclidean $d$-space, a Cartesian
coordinate system with axes parallel to the edges of $C_2$.
Therefore, there are $a_i^{(0)}, b_i^{(0)} \in \R$, such that
\[ u^{}_0(T)= \{ x \in E_0 \, | \, a_i^{(0)} \le x_i \le b_i^{(0)}
\} \]
where $b_i^{(0)} - a_i^{(0)} =2$.
Consider the tail $t(T)=(T=T_0,T_1,T_2,\ldots)$ and let
\[ u_0(T_j):= \supp(W(T_j)) \cap E_0, \quad (T_j\in t(T_0)) \]

Let us calculate a representation of $u_0(T_{j+1})$ in terms of
$u_0(T_j)$ and the biinfinite word $(\sigma_i)_{i \in I}$ encoding
$t(T)$. The intersection $u^{}_0(T_j)=\supp(W(T_j)) \cap E_0$ is a
$d$-cube
\[ u^{}_0(T_j)= \{ x \in E_0 \, | \, a_i^{(j)} \le x_i \le b_i^{(j)}
\} \] 
of edge-length $2^{j+1}$: $b_i^{(j)} - a_i^{(j)} =2^{j+1}$.
Now we  calculate coordinates of the section $u^{}_0(T_{j+1})$.
Since $W(T_j) \subseteq W(T_{j+1})$, we have $a_i^{(j+1)} \le
a_i^{(j)}$ and $b_i^{(j)} \le b_i^{(j+1)}$. Moreover, 
if  in an infinite word $s(T)=(\sigma^{(1)}, \sigma^{(2)},
\ldots)$ holds $\sigma_i^{(j+1)}=-1$, then
$a_i^{(j+1)}=a_i^{(j)}-2^{j+1}$, $b_i^{(j+1)}=b_i^{(j)}$. If
$\sigma_i^{(j+1)}=1$, then $a_i^{(j+1)}=a_i^{(j)}$ and
$b_i^{(j+1)}=b_i^{(j)}+2^{j+1}$. Altogether we obtain
\begin{equation}
  a_i^{(j+1)} = a_i^{(j)}+ \frac{\sigma_i^{(j+1)}-1}{2} ( b_i^{(j)}
- a_i^{(j)}),  \label{aik}
\end{equation}
\begin{equation}
  b_i^{(j+1)} = b_i^{(j)}+ \frac{\sigma_i^{(j+1)}+1}{2} ( b_i^{(j)}
- a_i^{(j)}).  \label{bik}
\end{equation}
Note again, that from (\ref{aik}) and (\ref{bik}), in particular,
it follows that  the intersection of the tower $W(T_{j})$ with the
horosphere $E_0$ is a $d$-cube of edge length $2^{j+1}$.
\begin{thm} \label{2kpools}
  In a B-tiling in $\H^{d+1}$ the following
  properties hold:
  \begin{enumerate}
  \item The number of pools is $2^k$ for some $0 \le k \le d$.
  \item For any $d$ and any $0 \le k \le d$, there are
    B-tilings in $\H^{d+1}$ with $2^k$ pools.
  \item Given $0\leq k\leq d$, in all B-tilings in $\H^{d+1}$ with $2^k$ pools
  the supports of all pools are pairwise congruent to each other.
  \item Given a B-tiling $T$, all pools
    are pairwise congruent to each other with respect to tiles.
  \item All $2^k$ pools in a $B$-tiling share a common
    $(d-k+1)$-plane.
  \end{enumerate}
\end{thm}
\begin{proof}
Let $\P$ be a pool in a B-tiling $\T$. By Proposition
\ref{pooltower}, $\P$ is the union of towers $(W(T_j))_{j \in
\N}$, where $W(T_j) \subset W(T_{j+1})$. Let $E_0$ be a horosphere
that contains some $b$-facet $b \subset T \in \P$. Since the
intersection of a tower and $E_0$ is a $d$-cube, the intersection
of $\P$ and $E_0$ is the union of countably many $d$-cubes
$\{\Box_j\}_{j \ge 0}$, where $\Box_j \subset \Box_{j+1}$.

Let $a_i^{(j)}, b_i^{(j)}$ as in (\ref{aik}),(\ref{bik})
correspond to $\P = \bigcup_{j \ge 0} W(T_j)$. From (\ref{aik})
and (\ref{bik}) we read off: if there are only finitely many $k$,
such that $\sigma_i^{(j)}=-1$, then there is a sharp lower bound
$a$ for $a_i^{(j)}$. Therefore, in this case the union $\bigcup_{j
\ge 0} \Box_j$ (where $\Box_j=W(T_j) \cap E_0$) is contained
in the half-space $H^+_i=\{ x=(x_1,\ldots,x_d) \in E_0=\E^d \, | \,
x_i \ge a \}$.

Analogously, if there are only finitely many $j$, such that
$\sigma_i^{(j)}=1$, then there is an upper bound $b$ for
$b_i^{(j)}$. In this case $\bigcup_{j \ge 0} \Box_j$ is contained
in the half-space $\tilde{H}^-_i=\{ x=(x_1,\ldots,x_d) \in \E^d \, |
\, x_i \le b \}$. We obtain:

{\bf (A)} For any fixed coordinate $i$, $\bigcup_{j \ge 0} \Box_j$ is
either unbounded in one direction or unbounded in both directions.

Assume  a  hyperplane $\bar{h}$ is the boundary
(at least partly) of a pool $\P$ in $\T$. Consider a
tile with a $c$-facet on $\bar{h}$. All tiles $T_j\in t(T)$ have
$c$-facets in the hyperplane $\bar{h}$. Now, since $\T$ is face-to-face, 
the reflection $\tau$ in the hyperplane $\bar{h}$ moves $T$ to some
tile $T'$ which shares with $T$ a common $c$-facet. Note 
that, by assumption, $T'$ belongs to another pool $\P'$. The tail
$t(T)$ moves under $\tau$ into $t(T')$. Therefore, the pool $\P$
moves under $\tau$ into pool $\P'$ too. And, in particular, the tail
of $T'$ --- determining $\T \cap \P'$ uniquely --- is obtained from
$t(T)$ by a reflection. Thus $s(T)$ differs from $s(T')$ only by
exchanging the sign in one coordinate.

Now, since $\T$ is face-to-face, the hyperplane $\bar{h}$,
which separates two pools, is a totally separating hyperplane.
Indeed, if $\bar{h}$ contains a common facet of two tiles in a
horospheric layer, it cannot dissect any other tile in the same
layer. Since, for being (part of) a boundary of two pools, $\bar{h}$
separates some adjacent tiles in any layer, it cannot dissect a
tile in tiling $\T$ at all. Therefore, for two tiles $T$ and $T'$
lying on different sides of $\bar{h}$, their tails have no tiles in
common. 

Now, by {\bf (A)}, and since any hyperplane in the boundary of some
pool is parallel to some $c$-facet, the hyperplanes in $\partial \P$
are pairwise orthogonal to each other. Thus they partition 
$\H^{d+1}$ into $2^k$ pools, proving the first point of the
theorem. It is a simple exercise to construct B-tilings with $2^k$
pools for all $0 \le k \le d$. In fact, one can use sequences $s(T)$,
where the numbers of $1$s in exactly $k$ coordinates is finite. 
This proves the second point of the theorem. Being pairwise orthogonal
to each other, all the $k$ hyperplanes share a common
$d-k+1$-dimensional plane which is also orthogonal to the
horospheres, proving the fifth point of the theorem.

Now we know the structure of the intersection of the pool $\P$
with the horosphere $E$:
\[ \P\cap E= \E^m \oplus E^{+\,(d-m)},   \]
where $E^{+\,(d-m)}$ denotes a $d-m$-dimensional 'octant', i.e,
the sum $ \R^+ \oplus \cdots \oplus \R^+$ of $d-m$
half-lines. Therefore, in all B-tilings with $2^k$ pools, all supports
of the pools are the same, which proves the third point of the
theorem.  Note that, up to here, two pools may have congruent
supports, but can be pairwise different as tilings.

Now we prove that all $2^k$ in some given B-tiling $\T$ are pairwise
congruent with respect  to tiles. We did prove it already for two
pools having some $c$-facet in common. But any two pools can
be linked by a chain of pools in which all sequel pools have
$d$-dimensional boundary in common. Therefore, the fourth point of
the theorem is proved.
\end{proof}

Let us emphasise that point (4) of Theorem \ref{2kpools} implies that
all tails in a given B-tiling are cofinally equivalent, up to
multiplying entire coordinates (that is, entire sequences
$(\sigma^{(j)}_i)_{j \in \N}$) by $-1$. This is stated precisely in the
following corollary.
\begin{corollary} \label{alltails}
Let $\T$ be a B-tiling, and let $T,\tilde{T} \in \T$. Denote their
sequences by $s(T)=(\sigma^{(j)}_1, \ldots, \sigma^{(j)}_d)_{j \in
  \N}$, $s(\tilde{T})=(\tilde{\sigma}^{(j)}_1, \ldots,
\tilde{\sigma}^{(j)}_d)_{j \in \N}$. Then there are $m,n \in \N$ such
that for each $i \le d$ holds   
\[ \forall j \in \N: \; \sigma^{(j+m)}_i =  \tilde{\sigma}^{(j+n)}_i
\quad \mbox{or} \quad  \forall j \in \N: \; \sigma^{(j+m)}_i =
- \tilde{\sigma}^{(j+n)}_i. \]
If $T$ and $\tilde{T}$ are contained in the same layer, then 
there is $m \in \N$ such that 
\[ \forall j \ge m: \; \sigma^{(j)}_i =  \tilde{\sigma}^{(j)}_i
\quad \mbox{or} \quad  \forall j \ge m: \; \sigma^{(j)}_i =
- \tilde{\sigma}^{(j)}_i. \]
\end{corollary}
\begin{proof} 
By Theorem \ref{2kpools}, all pools in $\T$ are congruent with respect
to tiles. This was proven by the fact, that any two pools are mapped
to each other by reflections $\tau_i$ in hyperplanes supporting the
boundary of the pools. This means that all tails are cofinally
equivalent after applying some of the reflections $\tau_i$. These
reflections act as multiplication by $-1$ in the $i$-th coordinate,
and the first claim follows. The second claim covers the special case
where the tiles are contained in the same layer. Then the tails
coincide from some common position $m$ on, up to reflections
$\tau_i$. 
\end{proof}

\section{\bf Symmetries of B\"or\"oczky Tilings in $\H^{d+1}$}
\label{h5} 

In the last section we strongly used the fact that horospheres in
$\H^{d+1}$ are isometric to $\E^d$, compare for instance
\cite[\S 4.7]{rat}. This will also be useful in the sequel, where we
apply the results of the last section to determine all possible
symmetries of a B-tiling. 

\begin{thm} \label{thmsymgrp}
Let $\T$ be a B\"or\"oczky tiling  in $\H^{d+1}$ with $2^k$
pools, let $0\leq k\leq d$, and denote by $\Sym(\T)$ its symmetry
group. Then $\Sym(\T)$ is isomorphic to
\begin{itemize}
\item $\Z \times B_k$ if there is a periodic sequence in $\T$, or
\item $B_k$ else;
\end{itemize}
where $B_k$ is  the symmetry group of a $k$-cube.
\end{thm}
The notation $B_k$ follows Coxeter, see \cite{h}. $B_k$ is the
group with Coxeter diagram
\begin{center}
\epsfig{file=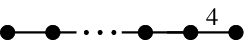, width=40mm}
\end{center}

\begin{proof} Let $S$ be the union of all $k$ hyperplanes bounding
the $2^k$ pools in $\T$. In the following, $\varphi$ denotes always a
symmetry of $\T$, i.e., an isometry of $\H^{d+1}$ with the property
$\varphi(\T)=\T$. Then, in particular, $\varphi(S)=S$.

The proof is organised as follows: First we prove that a goup
isomorphic to $B_k$ is contained in $\Sym(\T)$. Then we show that  
there is $\varphi \in \Sym(\T)$ which maps some horosphere $E_0$ to
some horosphere $E_j \ne E_0$ if and only if there is a tail which is
cofinally periodic (Claims 1,2,3). Finally it is shown that any
$\varphi \in \Sym(\T)$ fixing some horosphere $E_0$ is element of
$\bar{B}_k$ (Claims 4,5,6).   

By Theorem \ref{2kpools}, the intersection of the $k$ hyperplanes in
$S$ is a plane $H^{d-k+1}$ which is orthogonal to the horospheres $E_j$
in $\H^{d+1}$. 
The intersection of $S\cap E_j$ for any $j$ consists of $k$
pairwise orthogonal hyperplanes $h_1\ ldots, h_k$ in $E_j=\E^d$. They
share a common $d-k$-dimensional Euclidean plane.

Since a reflection $\tau_i$ in any of these $k$ hyperplanes, as we
have seen, keeps the tiling $\T$ invariant, the group generated
by the reflections $\langle\tau_i\rangle\subseteq \Sym(\T)$.

The restrictions $\nu_i$ of the hyperbolic reflections $\tau_i$
onto the $E$ are Euclidean reflections of $E$ in hyperplanes
$h_i$. Besides these reflections there are reflections $ \nu_{ij}$
in bisectors of all  dihedral angles between hyperplanes $h_i$ and
$h_j$, which also keep the set $S$ invariant. A Coxeter group
generated by all $\nu_i$ and $\nu_{ij}$, $1\leq i,j\leq k$, is
exactly a group $B_k$ of the $k$-dimensional cube. Denote by
$\bar{B}_k$ a Coxeter group $\langle\tau_i,\, \tau_{ij}\rangle$
generated by corresponding reflections of $\H^{d+1}$ and show that
$\bar{B}_k\subseteq \Sym(\T)$.

Let $\P\subseteq \bar{h}_i^+\cap \bar{h}_j^+$ and $\bar{h}_{ij}$ the
bisector of the dihedral angle $\angle \bar{h}_i\bar{h}_j$. Let us
consider a tile $T\in \P$ which is 'inscribed' into the $\angle
\bar{h}_i\bar{h}_j$, that is, two $c$-facets of $T$ 
lie on the boundary hyperplanes  $\bar{h}_i$ and $\bar{h}_j$.
The bisector $\bar{h}_{ij}$ dissects horospheric $a$- and $b$-facets
of $T$, which are $d$-dimensional Euclidean cubes, into two parts.
It is clear that the reflection $\tau_{ij}$ moves $T$ into itself.
So, $\tau_{ij}(T)=T$ and, consequently, $\tau_{ij}(t(T))=t(T)$ and
$\tau_{ij}(\P)=\P$.  By the face-to-face property, this implies
that $\tau_{ ij} (\T)=\T$. Therefore,  $\bar{B}_k\subseteq
\Sym(\T)$.

We proceed by showing that there are no other symmetries, except
possibly the ones arising from shifts along some line orthogonal to
the horospheres $E_i$ (possibly followed by some $\nu \in
\bar{B}_k$). These symmetries correspond to the occurrence of the
infinite cyclic group.  

{\em Claim 1:} $\varphi(t(T)) = t(\varphi(T))$.\\
The set $\varphi(t(T)) = \{ \varphi(T), \varphi(T_1), \varphi(T_2),
\ldots \}$ is clearly a tail. Since, for each fixed B-tiling $\T$,
each tail is uniquely determined by its first element, the claim
follows. 

{\em Claim 2:} If $T$ and $\varphi(T)$ are in
different layers, then $s(T)$ is cofinally periodic. \\
Let $T$ and $\tilde{T}=\varphi(T)$ be in different layers. Then, by
Corollary \ref{alltails}, there are $m,n \in \N$ such that for their
sequences holds:  
\begin{equation} \label{eq:j+mj+n}
\forall j \in \N: \; \sigma^{(j+m)}_i =  \tilde{\sigma}^{(j+n)}_i
\quad \mbox{or} \quad  \forall j \in \N: \; \sigma^{(j+m)}_i =
- \tilde{\sigma}^{(j+n)}_i \; (1 \le i \le d). 
\end{equation}
Since $T$ and $\tilde{T}$ are in different layers, we have $m \ne n$.
Without loss of generality, let $m > n$. 
By Claim 1, $t(T)$ and $t(\tilde{T})$ are congruent. Therefore
$s(T)$ and $s(\varphi(T))$ are identical, up to multiplication of
entire sequences $(\sigma_i^{(j)})_{j \in \N}$ by $-1$.  Thus, for
each $1 \le i \le d$, 
\begin{equation} \label{eq:jj}
\forall j \in \N: \; \sigma_i^{(j)}= \tilde{\sigma}_i^{(j)} 
\quad \mbox{or} \quad \sigma_i^{(j)}= - \tilde{\sigma}_i^{(j)} 
\end{equation}
Let $k=m-n$. From \eqref{eq:j+mj+n} and \eqref{eq:jj} follows for each
$1 \le i \le d$: 
\[ \forall j \in \N: \; \sigma^{(j+k)}_i = \pm \tilde{\sigma}^{(j)}_i =
\left\{ \begin{array}{l} \pm \sigma^{(j)}_i \\ \mp \sigma^{(j)}_i \\
  \end{array} \right. \]
We obtain either $\sigma^{(j+k)}_i = \sigma^{(j)}_i$, or
$\sigma^{(j+k)}_i = - \sigma^{(j)}_i$. In the second case holds for 
$j \ge k$: $\sigma^{(j+k)}_i = \sigma^{(j-k)}_i$. In each case, $p_i=2k$
is a period of $\sigma_i$. Then, the lowest common multiple of all the
$p_i$ is a period of $s(T)$. So we can already conclude: If the
sequence of one pool is not periodic, any possible symmetry of $\T$
maps tiles onto tiles in the same layer. 

{\em Claim 3:} If there is a cofinally periodic sequence $s(T)$ for some
$T \in \T$, then there exists a symmetry $\varphi$ such that
$\varphi^k(\T)=\T$ for all $j \in \Z$. \\
Let $s(T)$ be a periodic sequence in $\T$. Then there is a tile $T$ with a
tail $t(T)$ which belongs to this sequence. This tail is infinite in one
direction. In the other direction we can extend it in any way we want (since
by passing from $T$ through a $b$-facet to another tile, we can choose
each of the $2^d$ tiles which are lying there). In particular, we can extend
$t(T)$ to a biinfinite horospheric path
$t_b=(\ldots, T_{-2},T_{-1},T_0=T,T_1,T_2,\ldots)$, which belongs to a
biinfinite  periodic sequence
$s_b=(\ldots,\sigma^{(-2)},\sigma^{(-1)},\sigma^{(0)},
\sigma^{(1)},\sigma^{(2)},\ldots)$. Let $k$ be the period of $s_b$, then it
holds
\[ \cdots = \sigma^{(i-2k)} = \sigma^{(i-k)}=   \sigma^{(i)} = \sigma^{(i+k)}=
\sigma^{(i+2k)} = \cdots \]
for every $i \in \Z$. So for every $i \in \Z$, the set
$t^{(i)}=\{T_i,T_{i+1},T_{i+2} \ldots \}$ is congruent to
$t^{(i+k)}=\{T_{i+k},T_{i+k+1}, T_{i+k+2}, \ldots \}$. In other words,
there is an isometry $\varphi$ such that $\varphi(t^{(i)})=t^{(i+k)}$.
Consequently, $\varphi(t_b)=t_b$, and therefore $\varphi^j(t_b)=t_b$ for every
$j \in \Z$. From Proposition \ref{pooltower} follows that $t_b$ determines
its pool uniquely. By Theorem \ref{2kpools}, this pool
determines the whole tiling. It follows $\varphi^j(\T)=\T$ for $j \in \Z$.

The deduced symmetry $\varphi$ is obviously a shift along some line
$\ell$. If the period $k$ is not prime, it is possible that there is
an essential period (see the end of Section \ref{h2}) smaller than
$k$. However, some power of some symmetry $\psi$ corresponding to such
an essentail period is a shift along a line: $\psi^m = \varphi$, with
$\varphi$ as above.

By considering the action of $\varphi$ on tails $t(T)$, it
follows that $\varphi$ maps the horospheres $E_i$ to $E_{i+k}$. 
Hence $\varphi^m \ne \varphi^j$ for $m \ne j$. This gives rise to
the occurrence of an infinite cyclic group in the symmetry group of
$\T$. 

It remains to consider symmetries which fix some (and thus each)
horosphere $E_i$ in the sequel. This is the same as requiring
$\varphi$ to fix some layer $\Rr$: $\varphi$ then fixes also the boundary
$\partial \Rr = E_i \cup E_{i+1}$. Since $\varphi$ maps $a$-facets to
$a$-facets and $b$-facets to $b$-facets, it fixes $E_i$ as well as
$E_{i+1}$.  

{\em Claim 4:} Let $T \in \T$. Every symmetry $\varphi$, where
$\varphi(T)$ and $T$ are in the same layer $\Rr$ and in the same pool,
has a fixed point in $\Rr$. \\
By Claim 1, $\varphi$ maps $t(T)$ to $\varphi(t(T)) =
t(\varphi(T))$. Since $T$ and $\varphi(T)$ are in the same pool and in
the same layer, these tails coincide from some position $k$ on. 
In particular, there is $T_k \in t(T)$ such that $\varphi(T_k) =
T_k$. Thus, by Brouwer's fixed point theorem, $\varphi$ fixes some
point in $T_k$. Moreover, by the symmetry of $T_k$, $\varphi$ fixes
some point $x_a$ in the $a$-facet of $T_k$, as well as some point
$x_b$ in some $b$-facet of $T_k$. Being an isometry, $\varphi$ fixes
the line $\ell'$ through  $x_a$ and $x_b$ pointwise. Therefore, the
intersection $\ell' \cap \Rr$ consists of fixed points of $\varphi$. 

In fact, these symmetries are those arising from the reflections
$\tau_{ij}$. 

{\em Claim 5:} Let $T \in \T$. Every symmetry $\varphi$, where
$\varphi(T)$ and $T$ are in the same layer $\Rr$ but in different pools
$\P, \P'$, has a fixed point in this layer.\\
This can be shown in analogy to the proof of the last claim. By
Theorem \ref{2kpools}, $\P \cap \P'$ intersect in a common
$(d-k+1)$-plane. Thus there are tiles $\tilde{T} \in \P, \; \tilde{T}'
\in \P'$ having $d-k+1$-dimensional intersection. By the face-to-face
property, this intersection is a $(d-k+1)$-face $F$. By the proof of
Theorem \ref{2kpools}, this is also true for any pair of tiles
$\tilde{T}_k \in t(\tilde{T}), \tilde{T}'_k \in t(\tilde{T}')$, where
$k \in \N$. By Proposition \ref{pooltower}, there is $k$ such that
also holds: $\tilde{T}_k \in t(T), \tilde{T}'_k \in t(\varphi(T))$. 
Now, similar as in the proof of the last claim, $\varphi$ fixes a
$(d-k+1)$-face $F' = \tilde{T}_k \cap \tilde{T}'_k$. Moreover,
$\varphi$ fixes some point $x_a \in F$ ($x_b \in F$), contained in the
intersection of $F$ with the $a$-facets ($b$-facets) of $\tilde{T}_k$
and $\tilde{T}'_k$. As above, $\varphi$ fixes
the line $\ell'$ through $x_a$ and $x_b$ pointwise. Therefore, the
intersection $\ell' \cap R$ consists of fixed points of $\varphi$. 

In fact, these symmetries are the ones corresponding to the
reflections $\tau_i$. 

{\em Claim 6:}  Let $\varphi \in \Sym(\T)$ fix some (and thus each)
horosphere $E_i$. Then $\varphi \in \bar{B}_k$.\\
By the construction of a B-tiling, $\T$ induces a cube tiling of
$E_i$ by $a$-facets. For any symmetry $\varphi$ which fixes both $\T$ 
and $E_i$, the restriction $\nu := \varphi|_{E_i}$ fixes this cube
tiling. By the last two claims, each such $\nu$ has a fixed point in
$E_i$. In particular, $\nu$ is not a translation. But, fixing a cube
tiling in $E_i=\E^d$, the set of all these $\nu$ form a
crystallographic group.  The crystallographic groups which fix some
cube tiling and which contain no translations are well known, see for
instance \cite{h}. These are exactly the subgroups of $B_k$, which
proves Claim 6.

Altogether, we found two kinds of possible symmetries:
Symmetries in $\bar{B}_k$, and shifts along a line orthogonal to each
$E_i$, possibly followed by some map $\tau \in \bar{B}_k$. 
\end{proof}

A consequence of Theorem \ref{thmsymgrp} is that all
B\"or\"oczky-type tilings are non-cry\-stallo\-gra\-phic. This can
also be shown by Theorem \ref{crysthm}, along the same lines as in
Section \ref{sec:crystthm} for the 2-dimensional tiling: It is not
hard to convince oneself that the number of $k$-coronae in a
B\"or\"oczky-type tiling in $\H^{d+1}$ for $d > 2$ is strictly
larger than $2^{k-1}$ for $k \ge 2$, but a proper proof may be
lengthy, and yields no new result. 

As another consequence of Theorem \ref{thmsymgrp} we obtain:
\begin{corollary} Almost every B\"or\"oczky-type tiling has finite
  symmetry group.
\end{corollary}
\begin{proof}
Theorem \ref{thmsymgrp} shows, that $\Sym(\T)$ is infinite, if and
only if there is a tail $t(T)$ with a periodic sequence $s(T)$.
Since there are only countably many of these, there are only
countably many congruence classes of B\"or\"oczky-type tilings in
$\H^{d+1}$ with an infinite symmetry group. Since the
B\"or\"oczky-type tilings are non-crystallographic, it follows
from \cite{do2} that there are uncountably many congruence classes
of them, which proves the claim.
\end{proof}

\section*{Acknowledgements}
This work was started 2003 during the workshop 'Mathematics of
Aperiodic Order' at the Krupp-Kolleg in Greifswald and supported, in
part, by RFBR grants 06-01-72551 and 08-01-91202.

\end{document}